\documentclass[a4paper,reqno]{amsart}
\usepackage{amsmath}
\usepackage{amsfonts}
\usepackage{amssymb}
\usepackage{hyperref}
\usepackage{comment}
\usepackage{color}

\newtheorem{theorem}{Theorem}[section]
\newtheorem{lemma}[theorem]{Lemma}

\newtheorem{remark}[theorem]{Remark}
\newtheorem{example}[theorem]{Example}

\newcommand{\N}{\mathbb{N}}
\newcommand{\Z}{\mathbb{Z}}
\newcommand{\R}{\mathbb{R}}
\newcommand{\E}{\mathbb{E}}

\newcommand{\D}{\mathcal{D}}
\newcommand{\F}{\mathcal{F}}

\newcommand{\PP}{\mathbb{P}}

\DeclareMathOperator{\spt}{spt}

\begin{document}

\title{Fourier dimension of Mandelbrot multiplicative cascades}

\author[C. Chen] {Changhao Chen}
\address{CC: Center for Pure Mathematics, School of Mathematical Sciences, Anhui University, Hefei 230601, China}
\email{chench@ahu.edu.cn}

\author[B. Li] {Bing Li}
\address{BL: Department of Mathematics, South China University of Technology, Guangzhou, 510641, P.R. China}
\email{scbingli@scut.edu.cn}

 \author[V. Suomala] {Ville Suomala}
\address{VS: Research Unit of Mathematical Sciences, P.O.Box 8000, FI-90014, University of
Oulu, Finland}
\email{ville.suomala@oulu.fi}

\begin{abstract} 
We investigate the Fourier dimension, $\dim_F\mu$, of Mandelbrot multiplicative cascade measures $\mu$ on the $d$-dimensional unit cube. We show that if $\mu$ is the cascade measure generated by a sub-exponential random variable, then 
\[\dim_F\mu=\min\{2,\dim_2\mu\}\,,\]
where $\dim_2\mu$  is the correlation dimension of $\mu$ and it has an explicit formula. 
  
For cascades on the circle $S\subset\R^2$, we obtain
\[\dim_F\mu\ge\frac{\dim_2\mu}{2+\dim_2\mu}\,.\]
\end{abstract}


\maketitle

\section{Introduction}
The Fourier transform is a classical concept in harmonic analysis and its applications. More recently, it has gained importance also in geometric measure theory due to its various connections and application in dimension theory, see \cite{Mattila2015}.
 One reason for this is that the asymptotics of the Fourier transform lead to a useful concept of fractal dimension known as \emph{Fourier dimension}.

In this paper, we investigate the Fourier transform of a well-known family of random measures arising as limits of the \emph{Mandelbrot multiplicative cascades}. A program for investigating the Fourier decay of such random \emph{cascade measures} was initiated by Mandelbrot \cite{Mandelbrot1976} in 1976 and, somewhat more precisely, by Kahane \cite{Kahane1993} in 1993. The problem has, until now, remained open apart from the special case of fractal percolation which was settled some years ago by Shmerkin and Suomala \cite{ShmerkinSuomala2018}. We note that for related random models there has been recent progress in the Mandelbrot-Kahane program; In particular, Falconer and Jin \cite{FalconerJin2019} provide lower-bounds for the Fourier dimension of the planar Gaussian multiplicative chaos and Garban and Vargas \cite{GarbanVargas} obtained similar results for the Gaussian multiplicative chaos defined on the circle $S\subset\R^2$.
 
 In our main result, we show that for multiplicative cascades on the $d$-dimensional unit cube generated by a sub-exponential random variable, the Fourier dimension of the cascade measure is positive and, moreover, equal to the correlation dimension, provided the value of the latter is at most two. We are not aware of earlier examples of multifractal measures whose Fourier dimension is explicitly known, see also the discussion following Remark 1 in \cite{GarbanVargas}.  We also extend our analysis to cascades defined on the circle and provide a lower bound for the Fourier dimension of such spherical cascades.

For a measure $\nu$ on an Euclidean space 
$\R^d$, the Fourier transform of $\nu$ at $\xi\in \R^d$ is defined as
\[
\widehat{\nu}(\xi)=\int_{\R^n} e^{-2\pi i x \cdot\xi} d\nu(x).
\]
A measure $\nu$ is a \emph{Rajchman measure}, if $\widehat{\nu}(\xi)\longrightarrow 0$ as $|\xi|\to\infty$. The Rajchman property alone is an important regularity characterisation of the measure being intimately related to various analytic and geometric properties of the measure and its support, see e.g. the surveys \cite{Lyons1995, Sahlsten}. If $\nu$ is Rajchman, it is also of interest to investigate the speed of decay of $\widehat{\nu}(\xi)$ as $|\xi|\to\infty$.
The \emph{Fourier dimension}, $\dim_F$, quantifies the optimal polynomial decay, if any:
\[
\dim_F \nu =\sup\left\{0<s< d\,:\,|\widehat{\nu}(\xi)|=O(|\xi|^{-s/2})\right\}\,.
\]

The Mandelbrot multiplicative cascade is a popular model of random geometry exhibiting multifractality and stochastic self-similarity. It was introduced now 50 years ago by Mandlebrot \cite{Mandelbrot1974} to model turbulence. On the real-line, the cascade gives rise to a random metric with rich dimensional properties involving phase transitions, see e.g. \cite{BenjaminiSchramm2009, Barraletal2014, FalconerTroscheit2023}.
Besides many applications e.g. in random energy models and quantum gravity (see \cite{BenjaminiSchramm2009, Barraletal2013, Barraletal2014} for references), the cascade has an interpretation as a branching random walk, see e.g. \cite{Jaffuel2012, Aidekon2013, Madaule2017}. It is also related to the Gaussian multiplicative chaos whose Fourier decay has been investigated for planar domains in \cite{FalconerJin2019} and, more recently, in the circle \cite{GarbanVargas}. As mentioned earlier, the problem of determining the Fourier dimension of the cascade measure may be traced back to Mandelbrot \cite{Mandelbrot1976} and Kahane \cite{Kahane1993}.

We follow the classical approach of Kahane and Peyri\`ere \cite{KahanePeyriere1976} and define the multiplicative cascade as a sequence of absolutely continuous random measures. Their weak limit, the \emph{cascade measure}, and its Fourier dimension are our primary objects. 
Let $W\ge 0$ be a random variable with $\E(W)=1$. Denoting by $\D_n$ the collection of the dyadic sub-cubes of $[0,1]^{d}$ of side length $2^{-n}$, we attach an independent copy $W_Q$ of $W$ to each $Q\in\mathcal{D}_n$, $n\in\N$. If $Q_1,\ldots, Q_{n-1}$ are the dyadic ancestors of $Q_n\in\mathcal{D}_n$, let
\[\mu_n(x)=\prod_{i=1}^nW_{Q_i}\] for all $x\in Q_n$. By heart, $\mu_n$ is a picewise constant function $[0,1]^{d}\to[0,+\infty[$,  but it also gives rise to a measure via the interpretation $d\mu_n(x)=\mu_n(x)\,dx$.
It is well known and easy to see that the sequence $\mu_n$ converges in the weak* topology to a random limit measure which we denote by $\mu$.

\begin{remark}
    For notational simplicity, we carry out the details in the case of dyadic cascades. Using $b$-adic cubes for any $b\ge 3$ is also possible and makes no difference in the arguments. Our method also generalizes to Mandelbrot cascades generated by a $b^d$-tuple of non-neqative random variables $(W_1,\ldots,W_{b^d})$ under the condition $\E\left(W_k\right)=1$ for all $1\le k\le b^d$. (See e.g. \cite{ColletKoukiou1992, Molchan1996, FengBarral2018} for the precise definition.)
\end{remark}

The cascade measure $\mu$ is non-degenerate if
\begin{equation}\label{eq:crit}
\E(W\log W)<d
\end{equation}
 and otherwise $\mu\equiv 0$ almost surely, where $\log$ here and in what follows denotes logarithm in base two. The Kahane-Peyri\'ere condition \eqref{eq:crit} is our standing assumption throughout. 
 It is well known (see \cite{KahanePeyriere1976, Heurteaux2016}), that
 \begin{equation}\label{eq:dimloc}
 \dim(\mu,x)=\lim_{r\downarrow 0}\frac{\log \mu(B(x,r))}{\log r}=d-\E(W\log W)\,,
 \end{equation}
 for $\mu$-almost all $x\in[0,1]^d$. In particular,
 $\dim_H \mu=d-\E(W\log W)$
almost surely on non-extinction, where $\dim_H\mu$ denotes the Hausdroff dimension of $\mu$.

The Fourier decay of the Mandelbrot cascades is well understood in the following extreme case, the \emph{fractal percolation:}. Given $0<p\le 1$, let
\begin{equation}\label{eq:W_fp}
W=\begin{cases}
\frac1p\,,&\text{ with probability }p\\
0\,,&\text{ with probability }1-p
\end{cases}.
\end{equation}
The set $E=\spt\mu$ is the fractal percolation limit set and $\mu$ is the natural fractal percolation measure.
The fractal percolation is non-degenerate for $2^{-d}<p\le 1$ and
$
\dim_H E=\dim_H\mu=d+\log p
$
almost surely on $E\neq\varnothing$.
 Shmerkin and Suomala \cite{ShmerkinSuomala2018} showed that if $0\le d+\log p\le 2$ then $\mu$  is a Salem measure so that
\[\dim_F\mu=\dim_F E=\dim_H E=d+\log p\,,\]
almost surely on $E\neq\varnothing$. Here the notations $\dim_H E, \dim_F E$ refer to the Hausdorff and Fourier dimensions of the set $E$, see \cite[\S 3.6]{Mattila2015}. 

For any cascade measure, the support of $\mu$ is a fractal percolation set with the parameter $p=\PP(W>0)$, but it is easily seen from \eqref{eq:dimloc} that $\dim_H\mu<\dim_H\spt\mu$ unless $W$ is of the form \eqref{eq:W_fp}. This reflects the fact that while the fractal percolation is strongly monofractal (a.s. the local dimension $\dim(\mu,x)$
equals $d+\log p$ for all $x\in\spt\mu$), general cascade measures are multifractal so that the level sets of the map $x\mapsto\dim(\mu,x)$ have a rich dimensional structure \cite{Molchan1996, Barral2000, Heurteaux2016}.
In this paper, we extend the analysis of the Fourier dimension to general (multifractal) cascades on the unit cube as well as to cascades defined on the circle. Throughout the paper, we assume that the initial random variable $W$ is 
sub-exponential, that is,
\[\PP\left(|W|>t\right)\le 2\exp(-ct)\text{ for all }t>0\]
holds for some $c>0$. 
We note that this assumption is only used in the proof of the lower bounds for $\dim_F\mu$ in Theorem \ref{thm:main'}. The results concerning upper bounds for $\dim_F\mu$ are obtained from multifractal analysis and they hold under much weaker assumptions on $W$.

The correlation dimension of a finite measure $\nu$ on $[0,1]^d$ is defined as
\[
\dim_2 \nu =\liminf_{n\rightarrow \infty} \frac{\log \sum_{Q\in \D_n} \nu(Q)^{2}}{-n}\,.
\]
Alternatively, by \cite[Proposition 2.1]{HK1997},  this can be written as
\begin{equation*}
\dim_2\nu=\sup\left\{0\le s< d\,:\,\int\int  |x-y|^{-s} d\nu(x)d\nu(y)<\infty\right\}\,.
\end{equation*}
For any $0<s<d$, by \cite[Theorem 3.10]{Mattila2015}, there exists a constant $c(d, s)>0$ such that
\[
\int\int  |x-y|^{-s} d\nu(x)d\nu(y)=c(d, s)\int_{\R^d}|\widehat{\nu}(\xi)|^2|\xi|^{s-d}\,d\xi,
\]
revealing the generally valid estimate
\begin{equation}\label{eq:fourierL2}
\dim_F \nu \le \dim_2 \nu.
\end{equation}
For most measures, the above inequality is strict. If \eqref{eq:fourierL2} holds as an equality, and if the joint value of $\dim_F$ and $\dim_2$ is positive, this is an important regularity property for the measure. Note also that, apart from fractal percolation, the cascade measures satisfy $\dim_2\mu<\dim_H\mu$ reflecting the multifractality of $\mu$.

In our main theorem, we show that for a sub-exponential generating random variable, the Fourier dimension and the correlation dimension of the cascade measure agree if $\dim_2\mu\le 2$. Before stating the result, let us briefly discuss the proposed value of $\dim_2\mu$. For $p\ge 1$, let 
\[W_p=\frac{W^p}{\E(W^p)}\] 
and denote
\begin{equation} \label{eq:alpha-W}
\alpha(W) =
\begin{cases}
d-\log \E(W^2) & \text{ if } \E(W_2\log W_2)<d; \\
 \frac{2d(p-1)-2\log \E(W^p)}{p}   & \text{ if } \E(W_2\log W_2)\ge d,
\end{cases}
\end{equation}
where $1< p\le 2$ satisfies $\E(W_p\log W_p)=d$. We note that this value is unique, see the arXiv version of \cite{Jaffuel2012}.
Thus, $\alpha(W)$ is given by $d-\log\E(W^2)$ when the cascade corresponding to $W_2$ is subcritical. If $W_2$ gives rise to a supercritical cascade, $p$ is chosen so that $W_p$ is critical and then $\alpha(W)$ equals $(2d(p-1)-2\log\E(W^p))/p$.

\begin{theorem}\label{thm:main}
If $\mu$ is the cascade measure corresponding to a sub-exponential generating random variable $W$, then $\dim_2(\mu)=\alpha(W)$ and
\[\dim_F\mu=\min\{2,\alpha(W)\}\] 
almost surely on non-extinction. 
\end{theorem}

For the sake of concreteness, we present examples of Theorem \ref{thm:main} when applied to simple two point distributions $W$.
\begin{example}\label{ex:lambada} 
Given $0<\lambda<1$, we set
\begin{equation}
W=\begin{cases}
1+\lambda\,&\text{ with probability }\tfrac{1}{2}\\
1-\lambda\,&\text{ with probability }\tfrac{1}{2}
\end{cases}.
\end{equation}
Then almost surely $\dim_F \mu=\min\{2, d-\log (1+\lambda^2)\}$. In particular, $\dim_F \mu=2$ when $d\ge 3$.
\end{example}
\begin{proof}
By Theorem \ref{thm:main}, it is sufficient to show that $\E(W_2\log W_2)<1$, which is equivalent to
\begin{equation}\label{eq:lambda}
\E(W^2 \log W^2) <(1+\lambda^2)\log\left(2(1+\lambda^2)\right)\,.
\end{equation}
 Write  
$\E(W^2\log W^2)=\frac12(1+\lambda)^2\log ((1+\lambda)^2)+ \frac12(1-\lambda)^2 \log ((1-\lambda)^2)$.
Noting that $x\mapsto x\log x$ is increasing on $[1,\infty[$, we observe that $(1+\lambda)^2\log((1+\lambda)^2)\le 2(1+\lambda)^2\log(2(1+\lambda)^2)$. Since $(1-\lambda)^2\log((1-\lambda)^2)<0$, this verifies \eqref{eq:lambda}.
\end{proof}

In the above example, $E(W_2\log W_2)<1$ which allows a simple closed form expression for $\dim_F\mu$. For more general two-point distributions, both alternatives in \eqref{eq:alpha-W} are possible. 
\begin{example}
\begin{equation*}
W=\begin{cases}
3\,&\text{ with probability }\tfrac{1}{4}\\
\tfrac13\,&\text{ with probability }\tfrac{3}{4}
\end{cases}\,,
\end{equation*}
then $E(W_2\log W_2)>1$ and $E(W_p\log W_p)=1$ when $p\approx 1.162$. Applying Theorem \ref{thm:main} for $d=1$, we obtain $\dim_F(\mu)\approx 0.0301$ while $\dim_H\mu=1-\tfrac{\log 3}{\log 4}\approx 0.208$.
\end{example}

Garban and Vargas \cite{GarbanVargas} investigate the Fourier transform of random multifractal measures on the circle. They point out the lack of examples of such multifractal measures with non-trivial Fourier decay. Investigating the GMC on the circle, they show that the GMC measure is Rajchman for all relevant values of the parameter and provide a lower bound for the Fourier dimension for certain parameter values. Our analysis of the cascade measures can also be performed on the circle to provide quantitative lower bounds for the Fourier dimension. The cascade measure on the unit circle $S\subset\R^2$ is obtained as the push-forward $\widetilde{\mu}\circ f^{-1}$, where $f\colon[0,1]\to S$, $t\mapsto (\cos (2\pi t),\sin(2\pi t))$ and $\widetilde{\mu}$ is the cascade measure on the unit interval.  

\begin{theorem}\label{thm:circle}
Let $\mu$ be the cascade measure on the circle $S\subset\R^2$. Then
\[\frac{\alpha}{2+\alpha}\le \dim_F\mu\le \alpha\,,\]
almost surely on non-extinction, where $\alpha=\alpha(W)=\dim_2\mu$ is defined through \eqref{eq:alpha-W} for $d=1$.
\end{theorem}

This theorem follows by a straightforward modification of the proof of Theorem \ref{thm:main} and by the use of Van der Corput's lemma.  Although we present the details only for the circle, our method works also for certain cascades defined on more general surfaces in $\R^d$ with non-zero Gaussian curvature, providing lower bounds for their Fourier dimension.

The paper is organized as follows. In Section \ref{sec:prel}, we set up some notation and provide some auxiliary results. We also recall how the formulas in \eqref{eq:alpha-W} for the correlation dimension of $\mu$ are derived from the multifractal analysis of the cascades.
The Theorem \ref{thm:main} is proved in Section \ref{sec:proofmain} using the auxiliary tools from Section \ref{sec:prel} and utilizing the SI-martingale technique developed in \cite{ShmerkinSuomala2018}. Theorem \ref{thm:circle} is proved in the final Section \ref{sec:proofcircle}.

\begin{remark}
While preparing this article for publication, we learned that Chen, Han, Qiu, and Wang \cite{CHQW} have obtained related results on the Fourier dimension of Mandelbrot cascades using completely different methods. The paper \cite{CHQW} and this work are independent of each other and were made available on the arXiv on the same day.
\end{remark}

\section{Preliminaries}\label{sec:prel}

We denote by $\F_n$ the sigma-algebra generated by the random variables $W_J$, $J\in\cup_{j=1}^{n}\D_j$. We also denote $\mathcal{D}=\cup_n\mathcal{D}_n$. We use the familiar big-$O$ notation: For functions $f\ge 0$ and $g>0$, we write $f=O(g)$ if the ratio $\tfrac{f}g$ is bounded.  Depending on the context, the $O$ constants could be either random or deterministic. All logarithms are to base $2$.

We define the sub-exponential norm of a random variable $X$ as
\[||X||_{\Psi_1}=\inf\{c\ge 0\,:\,\E\left(\exp\left(|X|/c\right)\right)\le 2\}\,.\]
In our treatment of the cascade measures, we assume that $W$ is sub-exponential, i.e. $||W||_{\Psi_1}<\infty$
in order to apply the following Bernstein's inequality for sub-exponential random variables, see \cite[Theorem  2.8.1]{Vershynin2018}.
\begin{lemma}\label{lem:HoeffdingSG}
Let $X_{i}, i\in I$ be a sequence of independent sub-exponential
random variables with zero mean. Then, for all $\kappa>0$,
\[
\PP\left( \Big| \sum_{i\in I}   X_{i} \Big| >\kappa\right) \leq 2 \exp \left(-c \min\left\{\frac{\kappa^{2}}{\sum_{i\in I}||X_i||^2_{\Psi_1}}\,,\, \frac{\kappa}{\max_{i\in I}||X_i||_{\Psi_1}}\right\}\right)\,,
\]
where $c>0$ is an absolute constant.
\end{lemma}

Note that we allow $\PP(W=0)>0$ to cover also cascade measures with fractal support, see \eqref{eq:W_fp}. Despite \eqref{eq:crit}, it is thus possible that $\PP(\mu([0,1]^d)=0)>0$. Naturally, we are only interested in the event of \emph{non-extinction}, i.e. $\mu([0,1]^d)>0$. 
For technical reasons, we also exclude the trivial case $W\equiv 1$. 
We recall the following lemma, which in particular implies that the correlation dimension of $\mu$ equals $\alpha(W)$ almost surely on non-extinction.

\begin{lemma}\label{lem:dim_2}
Let  $\alpha(W)$  be given by  \eqref{eq:alpha-W}. Then, almost surely on non-extinction,
\[\lim_{n\to\infty}\frac{\log\sum_{Q\in\D_n}\mu_n(Q)^2}{-n}=\lim_{n\to\infty}\frac{\log\sum_{Q\in\D_n}\mu(Q)^2}{-n}=\alpha(W)\,.
\]    
\end{lemma}

\begin{proof}
The result may be derived straightforwardly from multifractal analysis on the cascades. For convenience, we provide an outline of the argument. Let $\alpha=\alpha(W)$.  
Let $q>1$ such that $\E(W_q\log W_q)<d$. 
Observe that 
\[Y_n(q):=2^{nd(q-1)}\E(W^q)^{-n}\sum_{Q\in\D_n}\mu_n(Q)^q\]
is a martingale.
Applying \eqref{eq:crit} to the cascade driven by $W_q$ implies that $Y_n(q)$
converges to a non-zero limit almost surely on non-extinction and this, in particular, implies that 
\begin{equation}\label{eq:dim_qn}
\lim_{n\to\infty}\frac{\log\sum_{Q\in\D_n}\mu_n(Q)^q}{-n}=\tau(q)   
\end{equation}
almost surely on non-extinction, where $\tau(q)=d(q-1)-\log\E(W^q)$. Moreover, also
\begin{equation}\label{eq:dim_q}
\lim_{n\to\infty}\frac{\log\sum_{Q\in\D_n}\mu(Q)^q}{-n}=\tau(q)    
\end{equation}
almost surely on non-extinction. In a special case, \eqref{eq:dim_q} was proved by Collet and Koukiou \cite{ColletKoukiou1992}. For an argument that applies in the present setting, and more generally, see e.g. Barral and Gon\c calves \cite[Theorem 3.1 (1)]{BarralConcalves2011}. If $\E(W_2\log W_2)< d$, the claim now follows by applying \eqref{eq:dim_qn}--\eqref{eq:dim_q} with $q=2$.

If $\E(W_2\log W_2)\ge d$, there exists a unique $1<p\le 2$ such that $\E(W_p\log W_p)=d$. For any $1<q<p\le 2$, by Jensen's inequality, we have 
\begin{align*}
\left(\sum_{Q\in\D_n}\mu_n(Q)^2 \right)^{q/2}\le \sum_{Q\in\D_n}\mu_n(Q)^q,
\end{align*}
and likewise for $\mu(Q)$ in place of $\mu_n(Q)$. Letting $q\nearrow p$ along a subsequence and combining with \eqref{eq:dim_qn}--\eqref{eq:dim_q}, this implies 
\begin{equation}\label{eq:lb}
\liminf_{n\to\infty}\frac{\log\sum_{Q\in\D_n}\mu_n(Q)^2}{-n}\ge\frac{2\tau(p)}{p}\,,\quad\liminf_{n\to\infty}\frac{\log\sum_{Q\in\D_n}\mu(Q)^2}{-n}\ge\frac{2\tau(p)}{p}\,.    
\end{equation}
On the other hand, if $1<q<q'<p$, the convexity of $q\mapsto\log\sum_{Q\in\D_n}\mu_n(Q)^q$ implies
\begin{align*}
&\frac{\log\sum_{Q\in\D_n}\mu_n(Q)^2}{-\log n}\\
&\le\frac{\log\sum_{Q\in\D_n}\mu_n(Q)^q}{-\log n}+\frac{\log\sum_{Q\in\D_n}\mu_n(Q)^{q'}-\log\sum_{Q\in\D_n}\mu_n(Q)^{q}}{-\log n}\times\frac{2-q}{q'-q}\,,
\end{align*}
and likewise for $\mu(Q)$ in place of $\mu_n(Q)$.
Using \eqref{eq:dim_qn}--\eqref{eq:dim_q} with the values $q,q'$ and letting first $q\nearrow q'$ and then $q'\nearrow p$, and noting that $\tau'(p)=\tau(p)/p$, implies
\begin{equation}\label{eq:ub}
\limsup_{n\to\infty}\frac{\log\sum_{Q\in\D_n}\mu_n(Q)^2}{-n}\le\frac{2\tau(p)}{p}\,,\quad\limsup_{n\to\infty}\frac{\log\sum_{Q\in\D_n}\mu(Q)^2}{-n}\le\frac{2\tau(p)}{p}\,,    
\end{equation}
see \cite[p. 693]{Molchan1996}. The claim follows by combining \eqref{eq:lb} and \eqref{eq:ub}.
\end{proof}

\begin{remark}
In an earlier version of the article, we used more recent results of A\"id\'ekon \cite{Aidekon2013} and Madaule \cite{Madaule2017} on additive martingales in the branching random walk to derive a weaker version of Lemma \ref{lem:dim_2}. We thank the referee for providing a simpler argument and for pointing us to the relevant literature.    
\end{remark}

\section{Proof of Theorem \ref{thm:main}} \label{sec:proofmain}

Recall that by \eqref{eq:fourierL2}, we have $\dim_F \mu \le \dim_2 \mu$. Thus,  Theorem \ref{thm:main} is a corollary of Lemma \ref{lem:dim_2} and the following two results.

\begin{theorem}\label{thm:main'}
$\dim_F\mu\ge\min\{2,\alpha(W)\}$
almost surely on non-extinction, where $\alpha(W)$ is given by \eqref{eq:alpha-W}.
\end{theorem}

\begin{theorem}\label{thm:dimF2}
    $\dim_F\mu\le 2$ almost surely.
\end{theorem}

We prove Theorems \ref{thm:main'} and \ref{thm:dimF2} in separate subsections. We begin by proving the lower bound for the Fourier dimension, which we consider to be our main result.

\subsection{Proof of Theorem \ref{thm:main'}}

Towards the proof of $\dim_F\mu\ge\min\{2,\alpha(W)\}$, we
estimate
$$
\widehat{\mu_n}(\xi)=\int_{[0,1]^{d}} e^{-2\pi i x\cdot\xi}\mu_n(x)\,dx=\sum_{Q\in \D_n}\int_Q e^{-2\pi i x\cdot\xi}\mu_n(x)\,dx\,.
$$
For any $n\in\N$ and $\xi\in\R^d$,
$$
\widehat{\mu_{n+1}}(\xi)-\widehat{\mu_{n}}(\xi)=\sum_{Q\in D_n} X_{Q,\xi}\,,
$$
where
$$
X_{Q,\xi}=\int_Q (\mu_{n+1}(x)-\mu_n(x))e^{-2\pi i x\cdot\xi}dx.
$$
Conditional on $\F_n$, the random variables $X_{Q,\xi}, Q\in \D_n$, are independent and have zero mean.
Moreover, their real and imaginary parts are sub-exponential and satisfy
\begin{equation}\label{eq:sGnormbound}
||\Re(X_{Q,\xi})||_{\Psi_1}, ||\Im(X_{Q,\xi})||_{\Psi_1}\le C \mu_n(Q)\,.
\end{equation}
Indeed, noting that $W-1$ is also sub-exponential and letting $Q_1,\ldots,Q_{2^d}\in\D_{n+1}$ be the dyadic child cubes of $Q$,
\begin{align*}
\left|\Re(X_{Q,\xi})\right|&=\left|\int_{Q} (\mu_{n+1}(x)-\mu_n(x) )\cos (2\pi x\cdot\xi)\,dx\right|\\
&\le\sum_{k=1}^{2^d}\int_{Q_k}|W_{Q_k}-1|\,\mu_n(x)\,dx\\
&\le\mu_n(Q)\sum_{k=1}^{2^d}|W_{Q_k}-1|
\end{align*}
and likewise for $\Im(X_{Q,\xi})$. The estimates in \eqref{eq:sGnormbound} now follows using the triangle inequality for the sub-exponential norm.

Let $\varepsilon>0$. For each $n\in \N$, denote $S_n=\sum_{Q\in\D_n}\mu_n(Q)^2$. Noting \eqref{eq:sGnormbound}, we have
$\sum_{Q\in D_n} ||X_{Q,\xi}||^2_{\Psi_1}=O(S_n)$ and also $\max_{Q\in\D_n}||X_{Q,\xi}||_{\Psi_1}=O(1)S^{1/2}_n$. Thus
\[
\min\left\{\frac{2^{2\varepsilon n} S_n}{\sum_{Q\in\D_n}||X_{Q,\xi}||^2_{\Psi_1}}\,,\, \frac{2^{\varepsilon n}S_n^{1/2}}{\max_{Q\in \D_n}||X_{Q,\xi}||_{\Psi_1}}\right\}\ge c 2^{\varepsilon n}\,.
\]
Lemma \ref{lem:HoeffdingSG} then implies that for any fixed $\xi\in\R^d$,
\begin{equation*}
\PP\left(|\widehat{\mu_{n+1}}(\xi)-\widehat{\mu_{n}}(\xi)|>2^{\varepsilon n}S_{n}^{1/2}\,|\, \mathcal{F}_n \right)\le C \exp(-c2^{n\varepsilon})\,,
\end{equation*}
and hence also unconditionally
\begin{equation*}
\PP\left(|\widehat{\mu_{n+1}}(\xi)-\widehat{\mu_{n}}(\xi)|>2^{\varepsilon n}S_{n}^{1/2} \right)\le C \exp(-c2^{n\varepsilon})\,.
\end{equation*}

For $\xi\in \R^d$, we use the notation $|\xi|_\infty=\max\{|\xi_1|, \ldots, |\xi_d|\}$ for the $\ell^\infty$-norm. Let $\tau<\min\{2,\alpha(W)\}$ and let $B_n$ denote the event
$$
|\widehat{\mu_{n+1}}(\xi)-\widehat{\mu_n}(\xi)|>2^{\varepsilon n}S_{n}^{1/2}\text{ for some }\xi\in2^{-n\tau/2}\Z^d\,,\, |\xi|_\infty\le 2^{n+1}\,.
$$
Then 
\begin{equation}\label{eq:B_n}
    \PP(B_n)\le C2^{nd(1+\tau/2)}\exp(-c2^{n\varepsilon})\,,
\end{equation} 
and thus by the Borel-Cantelli lemma,
\begin{equation}\label{eq:B}
\PP(B_n\text{ infinitely often})=0\,.
\end{equation}

By Lemma \ref{lem:dim_2}, 
for each $\beta<\alpha(W)$, $S_n=O(2^{-n\beta})$ almost surely. Adjusting $\beta$ and $\varepsilon$ and combining with \eqref{eq:B}, we infer that there is a random constant $K<\infty$ such that
\begin{equation*}
|\widehat{\mu_{n+1}}(\xi)-\widehat{\mu_n}(\xi)|\le K 2^{-n\tau/2}
\end{equation*}
holds for all $n\in\N$ and $\xi\in 2^{-n\tau/2}\Z^d$, $|\xi|_\infty\le 2^{n+1}$.
Also, since $\mu_n([0,1]^d)$ is a martingale with $\E(\mu_1([0,1]^d))=1$, it follows that $\sup_n \mu_n([0,1]^d)<\infty$ almost surely. Whence, the maps
\[
\xi\mapsto\widehat{\mu_{n+1}}(\xi)-\widehat{\mu_{n}}(\xi)
\]
are Lipschitz with a Lipschitz constant independent of $n$, see \cite[(3.19)]{Mattila2015}. Taking this into account, we observe that
\begin{equation}\label{eq:prob-part}
|\widehat{\mu_{n+1}}(\xi)-\widehat{\mu_n}(\xi)|\le M 2^{-n\tau/2}\,.
\end{equation}
for all $|\xi|_\infty\le 2^{n+1}$, where $M<\infty$ is random but independent of $n$ and $\xi$.

Assuming \eqref{eq:prob-part}, we may now complete the proof as follows: If $|\xi|_\infty >2^{n+1}$, we may write $\xi=2^{n+1}q+\xi'$ where $q\in\Z^d$ and $|\xi'|_\infty<2^{n+1}$. By an elementary computation (see \cite[Proof of Theorem 14.1]{ShmerkinSuomala2018}),
$$
|\widehat{\mu_{n+1}}(\xi)-\widehat{\mu_{n}}(\xi)|\le  \frac{2^{n+1}}{|\xi|_\infty}  \left|\widehat{\mu_{n+1}}(\xi')-\widehat{\mu_{n}}(\xi')\right|\,.
$$
It follows that
\begin{equation*}
|\widehat{\mu_{n+1}}(\xi)-\widehat{\mu_{n}}(\xi)|\le M\min\left\{1, \frac{2^{n+1}}{|\xi|_\infty} \right\}2^{-n\tau/2},\quad \forall n\in\N, \xi\in\R^d\,.
\end{equation*}

Finally, for any $\xi\in \R^d$, picking $n_\xi$ so that $2^{n_\xi}<|\xi|_\infty\le 2^{n_\xi+1}$, we have
\begin{align*}
|\widehat{\mu_{m}}(\xi)-\widehat{\mu_0}(\xi)|&\le\sum_{n=0}^{n_\xi}|\widehat{\mu_{n+1}}(\xi)-\widehat{\mu_n}(\xi)|+\sum_{n=n_\xi+1}^{m-1}|\widehat{\mu_{n+1}}(\xi)-\widehat{\mu_n}(\xi)|\\
&=|\xi|^{-1}\sum_{n=0}^{n_\xi}O\left(2^{n(1-\tau/2)}\right)+\sum_{n=n_\xi+1}^{m-1}O\left(2^{-n\tau/2}\right)\\
&=O(|\xi|^{-\tau/2})\,,
\end{align*}
for all $m\ge n_\xi$, using the fact $\tau<2$ to bound the first summand.
Noting that $\widehat{\mu_0}(\xi)=O(|\xi|^{-1})$, we have shown that for all $\xi\in \R^d$,
\[
\widehat{\mu}(\xi)=\lim_{m\rightarrow \infty} \widehat{\mu_m}(\xi)=O(|\xi|^{-\tau/2})+O(|\xi|^{-1})=O(|\xi|^{-\tau/2})\,.
\]
Since  $\tau<\min\{2,\alpha(W)\}$ is arbitrary, we conclude that  $\dim_F\mu\ge \min\{\alpha(W), 2\}$ almost surely on non-extinction.

\subsection{Proof of Theorem \ref{thm:dimF2}}
We base our proof on the fact that for any measure $\nu$ with $\dim_F \nu>2$, all of its orthogonal projections onto lines are absolutely continuous with a continuous density, see \cite[Proposition 3.2.12]{Grafakos08}. The projections of $\mu$ onto the principal directions have been studied in depth in \cite{FengBarral2018}, where, in particular, it is shown that these projections are absolutely continuous if $\E(W\log W)<d-1$ while for $\E(W\log W)\ge d-1$, these projections are singular with respect to the Lebesgue measure. In the following we show that, a.s. on non-extinction, the principal projections cannot have continuous densities.

To that end, let $d\ge 2$ and let $\pi$ denote the orthogonal projection onto the first coordinate axis which we identify with $\R$. We consider the following auxiliary random variables:
\begin{align*}
X_n&=2^n \mu\left(\pi^{-1}([0,2^{-n}])\right)\,.
\end{align*}
We will make use of the following lemma.
\begin{lemma}\label{lem:X}
    \begin{enumerate}
        \item The finite limit $X=\lim_n X_n$ exists almost surely.\label{eq:lem1}
        \item $\PP(X=x)<1$ for all $x>0$.\label{eq:lem2}
        \item If $X=0$ almost surely, then $\dim_F\mu\le 2$ almost surely.\label{eq:lem3}
    \end{enumerate}
\end{lemma}
\begin{proof}
The key to \eqref{eq:lem1} is the observation that if the value of $\mu(Q)$ is known for some dyadic cube, and if $Q$ is split in half along a co-ordinate plane, then the expected measure of both halves is $\mu(Q)/2$. In more detail, let $\mathcal{E}_n$ denote the sigma-algebra generated by the random variables $\mu(Q)$, $Q\in\mathcal{D}_n$. Then each $X_n$ is $\mathcal{E}_n$-measurable. We derive the claim \eqref{eq:lem1} from the martingale convergence theorem by showing that $X_n$ is a martingale with respect to the filtration $\mathcal{E}_n$. To that end, let 
$Y_n=2^n \mu\left(\pi^{-1}([2^{-n},2^{-n+1}])\right)$ and let $F_n\colon[0,1]^d\to[0,1]^d$ be a map which is a reflection with respect to the plane $\pi^{-1}(2^{-n})$ on $\pi^{-1}([0,2^{-n+1}])$ and the identity elsewhere. Since the random variables $W_Q$, $Q\in\cup_{m\in\mathbb{N}}\mathcal{D}_m$, are independent and identically distributed, and each $F_n$ preserves the nested structure of the dyadic cubes, it follows that for each given $n$, the law of (the random variables $\mu_m$, $m\in\mathbb{N}$, and thus also of) $\mu$ are invariant under $F_n$. Moreover, since $F_n(Q)=Q$ if $Q\in\mathcal{D}_{n-1}$, each $\mathcal{E}_{n-1}$-measurable random variable is $F_n$-invariant and thus the law of $\mu$ is $F_n$-invariant also conditional on $\mathcal{E}_{n-1}$. Since $Y_n$ may be expressed as $Y_n=2^n \nu\left(\pi^{-1}([0,2^{-n}])\right)$ for $\nu=\mu\circ F_n$, we conclude that $\E(Y_n|\mathcal{E}_{n-1})=\E(X_n|\mathcal{E}_{n-1})$ and since
\begin{align*}
    2^{-n}X_n+2^{-n}Y_n=\mu\left(\pi^{-1}([0,2^{-n-1}])\right)=2^{-n-1}X_{n-1}\,,
\end{align*}
    it follows that 
    \[\E(X_n|\mathcal{E}_{n-1})=\frac12\left(\E(X_n|\mathcal{E}_{n-1})+\E(Y_n|\mathcal{E}_{n-1})\right)=X_{n-1}\,.\]
    
To prove \eqref{eq:lem2}, observe that, for all $n\in \N$, the following identity holds in distribution:
\begin{equation*}
X_{n+1}\overset{\mathbf{d}}{=} 2^{1-d}\sum_{j=1} ^{2^{d-1}} W(j)X_n(j)\,, 
\end{equation*}
where $W(j)$, $X_n(j)$, $j=1, \dots 2^{d-1}$, are independent and the $W(j)$, $X_n(j)$ are distributed according to $W$, $X_n$, respectively.
Letting $n\to\infty$, this implies that also 
\begin{equation}\label{eq:self-similar}
X\overset{\mathbf{d}}{=} 2^{1-d}\sum_{j=1} ^{2^{d-1}} W(j)X(j)\,, 
\end{equation}
where $W(j)$, $X(j)$, are independent and the $W(j)$, $X(j)$ are distributed according to $W$, $X$, respectively. Since $W$ is not a constant a.s., it is obvious that \eqref{eq:self-similar} cannot be satisfied if $X$ is almost surely a non-zero constant.

Let $h(x)$ denote the density of $\mu\circ\pi^{-1}$. To prove \eqref{eq:lem3}, we may assume it exists for all $x\in\mathbb{R}$. For $q\in (2^{-m}\N)\cap[0,1]$ and conditional on $\mathcal{F}_m$, $h(q)$ is a linear combination of independent copies of $X$. 
Thus, if $X=0$ almost surely, then a.s. $h(q)=0$ for all $q\in\mathbb{Q}\cap[0,1]$. Whence
    \[\PP\left(h\text{ is continuous and not identically zero}\right)=0\,.\]
    Consequently, if $X=0$ almost surely, then $h$ is discontinuous (or not everywhere defined) a.s. on non-extinction and thus $\dim_F\mu\le 2$ almost surely. 
\end{proof}

\begin{remark}
    Although it is not needed here, it can be shown that $X=0$ almost surely if and only if $\dim_H\mu\le 1$ almost surely or, equivalently, if the cascade measure in $[0,1]^{d-1}$ constructed with $W$ goes extinct almost surely, that is, if $\E(W\log W)\ge d-1$. 
\end{remark}

Noting that $h(t)=0$ for all $t<0$, it immediately follows from Lemma \ref{lem:X} that $\dim_F\mu\le 2$ with positive probability. For the sake of rigour, we provide the details on how to improve ‘positive probability' to ‘almost surely'.
Recall that our aim is to verify that, almost surely on non-extinction, the density of the projection $\mu\circ \pi^{-1}$ cannot exist as a continuous function on $\R$. The proof proceeds by showing that, for all dyadic rationals $q\in(2^{-m}\N)\cap[0,1]$, the probability that the density exists and is continuous at $q$ is uniformly bounded away from one, conditional on $q\in\spt\mu_m\circ\pi^{-1}$. 

To that end, by Lemma \ref{lem:X} \eqref{eq:lem3}, we may assume that $\PP(X=0)<1$. Let
\[c=\sup_{0\le x<\infty}\PP(X=x)\,.\]
Combining with Lemma \ref{lem:X} \eqref{eq:lem2}, we know that $c<1$. 
Let $I_L,I_R\in\D_m(\R)$ be consecutive dyadic intervals and let $q$ be their common endpoint. 
Let  $Q_1,\ldots, Q_{2^{(d-1)m}}$ be an enumeration of the cubes $Q\in\D_m$ for which  $Q\subset\pi^{-1}(I_R)$. Conditional on $\mu_m(\pi^{-1}(I_R))>0$, let
\[j=\min\left\{1\le i\le 2^{m(d-1)}\,:\,\mu_m(Q_i)>0\right\}\,.\]
Denote
$S_L=\pi^{-1}(I_L)$, $S_R=\pi^{-1}(I_R)$, and $S_{R,j}=\cup_{i=j+1}^{2^{(d-1)m}}Q_i$. Let 
\[Z_n=2^{n}\mu\left(Q_j\cap\pi^{-1}([q,q+2^{-n}]\right)\,.\]
Then, conditional on $\mathcal{F}_m$, $Z_n$ is distributed as $2^{m}\mu_m(Q_j)X_{n-m}$ for $n\ge m$. 
Since
\begin{align*}
&\mu\left(\pi^{-1}([q,q+2^{-n}]\right)=2^{-n}Z_n+\mu\left(S_{R,j}\cap\pi^{-1}([q,q+2^{-n}])\right)\,,\\
&\mu\left(\pi^{-1}([q-2^{-n},q]\right)=\mu\left(S_L\cap\pi^{-1}\left([q-2^{-n},q]\right)\right)\,,    
\end{align*}
we observe that if $Z_n$ has a limit, the density of $\mu\circ\pi^{-1}$ may be continuous at $q$ only if 
\begin{equation}\label{eq:Z_n}
\lim_{n\to\infty} Z_n=\lim_{n\to\infty} 2^{n}\left(\mu\left(S_L\cap\pi^{-1}\left([q-2^{-n},q]\right)\right)-\mu\left(S_{R,j}\cap\pi^{-1}([q,q+2^{-n}])\right)\right).\end{equation}
Conditioning on $j$ and on the sigma-algebra generated by the random variables $\mu_m(Q_j)$ and $\mu_{n}|_{S_L\cup S_{R,j}}$, $n\ge m$, determines the value of the right-hand side of \eqref{eq:Z_n}. Since $Z_n$, $n\ge m$, and $\mu\left(S_L\cap\pi^{-1}\left([q-2^{-n},q]\right)\right)-\mu\left(S_{R,j}\cap\pi^{-1}([q,q+2^{-n}])\right)$ are independent conditional on $\mathcal{F}_m$, we may use
Lemma \ref{lem:X} to conclude that
\begin{align}\label{eq:density_Q}
\PP\left(\text{The density of }\mu\circ\pi^{-1}\text{ is continuous at }q\,|\,\mu_m(\pi^{-1}(I_R))>0\right)\le c\,.
\end{align}
Let $m\in\N$. Conditional on $\mathcal{F}_m$, the events 
\[\text{The density of }\mu\circ\pi^{-1}\text{ is continuous at }q\,,\]
are independent for $q\in(2^{1-m}\N)\cap[0,1]$. 
Using the Borel-Cantelli lemma and \eqref{eq:density_Q} now implies that the event
\[\text{The density of }\mu\circ\pi^{-1}\text{ is continuous,}\]
has probability zero conditional on non-extinction.

\begin{remark}
    The assumption that $W$ is sub-exponential is not needed in the proof of Theorem \ref{thm:dimF2} and the above proof, in fact, works under the most general assumptions, i.e. $W\ge 0$ and $\E(W)=1$.
\end{remark}

\section{Spherical cascades}\label{sec:proofcircle}

In this final section, we prove Theorem \ref{thm:circle}. The following lemma is standard, see \cite[Theorem 14.3]{Mattila2015}. We denote by $\sigma$ the surface measure on the unit circle $S\subset\R^2$.

\begin{lemma}\label{lem:Van}
Let $J \subset S$ be an arc. Then
\begin{equation}
\left | \int_{J} e^{-2\pi i x\cdot \xi} d\sigma(x) \right|=O(|\xi|^{-1/2}), \quad \forall \xi \in \R^2\,,
\end{equation}
where the $O$-constant is independent of $J$.
\end{lemma}

Recall that the cascade measure $\mu$ on the unit circle $S$ is defined as the push-forward $\widetilde{\mu}\circ f^{-1}$, where $f\colon[0,1]\to S$, $t\mapsto (\cos (2\pi t),\sin(2\pi t))$ and $\widetilde{\mu}$ is the cascade measure on the unit interval.

\begin{proof}[Proof of Theorem \ref{thm:circle}]

Since the correlation dimension is obviously invariant under $f$, from Lemma \ref{lem:dim_2} we infer that $\dim_2\mu=\dim_2\widetilde{\mu}=\alpha(W)$ almost surely on non-extinction. Moreover, also 
\[\lim_{n\rightarrow \infty} \frac{\log \sum_{I\in\widetilde{\D_n}} \mu_n(I)^2}{-n}=\alpha(W)\,,\]
where $\widetilde{\D_n}=f(\D_n)$ denotes the level $n$ dyadic arcs on $S$.

It remains to verify the lower bound for the Fourier dimension. The probabilistic part of the proof is very similar to the proof of Theorem \ref{thm:main} with only minimal changes: Let $\tau<\alpha(W)$. Introducing a parameter $R\ge 1$, we argue as in the proof of \eqref{eq:prob-part} and infer that almost surely 
\begin{equation}\label{eq:Fbnd}
|\widehat{\mu_{n+1}}(\xi)-\widehat{\mu_n}(\xi)|=O(1) 2^{-\tau n/2}, \quad \forall |\xi|\le 2^{nR}\,,
\end{equation}
where the $O$-constant is random but independent of $n$ and $\xi$. Indeed, the parameter $R$ only affects the exponent of $2^n$ in the upper bound for $\PP(B_n)$, see \eqref{eq:B_n}.

Using Lemma \ref{lem:Van}, we infer
\begin{equation}\label{eq:corput}
\begin{split}
\widehat{\mu_n}(\xi)&= \sum_{I\in \widetilde{\D_n}} \int_I e^{-2\pi i x\cdot\xi}\mu_n(x)\,d\sigma(x)
\\
&=O(1)2^n|\xi|^{-1/2}\sum_{I\in\widetilde{\D_n}}\mu_n(I)\\
&=O(1)2^n|\xi|^{-1/2}\,,
\end{split}
\end{equation}
where the last estimate holds since the total mass martingale $\mu_n(S)$ is almost surely bounded.

Note that the estimate \eqref{eq:corput} holds simultaneously for all $\xi\in\R^2$ and, likewise, \eqref{eq:Fbnd} is valid for all $|\xi|\le 2^{nR}$. For each $\xi\in\R^2$, we pick $n_\xi$ so that $2^{n_\xi R}\le|\xi|<2^{(n_\xi+1)R}$. Combining \eqref{eq:corput} and \eqref{eq:Fbnd} and telescoping with $R=2+\tau$ implies
\begin{align*}
|\widehat{\mu_{m}}(\xi)|&\le|\widehat{\mu_{n_\xi}}(\xi)|+\sum_{n=n_\xi}^{m}|\widehat{\mu_{n+1}}(\xi)-\widehat{\mu_n}(\xi)|\\
&=O\left (2^{n_\xi} |\xi|^{-1/2} \right )+\sum_{n=n_\xi}^{m} O\left(2^{-\tau n/2}\right)\\
&=O\left (2^{n_\xi}\right) |\xi|^{-1/2}+O(2^{-n_\xi\tau/2})\\
&=O(|\xi|^{-\tau/(4 +2\tau)})\,,
\end{align*}
for all $m\ge n_\xi$. Since $\tau<\alpha(W)$ is arbitrary, letting $m\to\infty$, this implies that $\dim_F\mu\ge\alpha(W)/(2+\alpha(W))$ almost surely on non-extinction and thus completes the proof.
\end{proof}

\begin{remark}
For certain monofractal constructions including the fractal percolation, Ryou \cite{Ryou2024} has shown that the natural measures, when lifted to the parabola, are Salem. It seems plausible that his method may be used to improve on the lower bound in Theorem \ref{thm:circle}.
\end{remark}

\section*{Acknowledgements} We are grateful to Meng Wu for bringing this question to our attention, as well as for many useful discussions on the topic. We thank an anonymous referee for a professional and very useful review that led to many improvements.

This work is dedicated to the memory of Junxian Li, the beloved father of Bing Li.

\subsection*{Financial support} C.C. and B.L.~were partially supported by National Key R\&D Program of China (No.2024YFA1013700), NSFC (No.12101002,12271176) and Guangdong Natural Science Foundation 2024A1515010946.

\subsection*{Competing interests} We affirm that there are no conflicts of interest pertaining to this manuscript.

\subsection*{Data Availability}
Data availability does not apply to this work, as no data was created or analysed in this study.


\begin{thebibliography}{10}

\bibitem{Aidekon2013}
E.~A\"id\'ekon.
\newblock Convergence in law of the minimum of a branching random walk.
\newblock {\em Ann. Probab.}, 41(3A):1362--1426, 2013.

\bibitem{Barral2000}
J.~Barral.
\newblock Continuity of the multifractal spectrum of a random statistically self-similar measure.
\newblock {\em J. Theoret. Probab.}, 13(4):1027--1060, 2000.

\bibitem{FengBarral2018}
J.~Barral and D.-J. Feng.
\newblock Projections of planar {M}andelbrot random measures.
\newblock {\em Adv. Math.}, 325:640--718, 2018.

\bibitem{BarralConcalves2011}
J.~Barral and P.~Gon\c{c}alves.
\newblock On the estimation of the large deviations spectrum.
\newblock {\em J. Stat. Phys.}, 144(6):1256--1283, 2011.

\bibitem{Barraletal2013}
J.~Barral, X.~Jin, R.~Rhodes, and V.~Vargas.
\newblock Gaussian multiplicative chaos and {KPZ} duality.
\newblock {\em Comm. Math. Phys.}, 323(2):451--485, 2013.

\bibitem{Barraletal2014}
J.~Barral, A.~Kupiainen, M.~Nikula, E.~Saksman, and C.~Webb.
\newblock Critical {M}andelbrot cascades.
\newblock {\em Comm. Math. Phys.}, 325(2):685--711, 2014.

\bibitem{BenjaminiSchramm2009}
I.~Benjamini and O.~Schramm.
\newblock K{PZ} in one dimensional random geometry of multiplicative cascades.
\newblock {\em Comm. Math. Phys.}, 289(2):653--662, 2009.

\bibitem{CHQW}
X.~Chen, Y.~Han, Y.~Qiu, and Z.~Wang.
\newblock {H}armonic {A}nalysis of {M}andelbrot {C}ascades in the context of vector-valued measures.
\newblock Preprint, available at arXiv:2409.13164.

\bibitem{ColletKoukiou1992}
P.~Collet and F.~Koukiou.
\newblock Large deviations for multiplicative chaos.
\newblock {\em Comm. Math. Phys.}, 147(2):329--342, 1992.

\bibitem{FalconerJin2019}
K.~Falconer and X.~Jin.
\newblock Exact dimensionality and projection properties of {G}aussian multiplicative chaos measures.
\newblock {\em Trans. Amer. Math. Soc.}, 372(4):2921--2957, 2019.

\bibitem{FalconerTroscheit2023}
K.~J. Falconer and S.~Troscheit.
\newblock Box-counting dimension in one-dimensional random geometry of multiplicative cascades.
\newblock {\em Comm. Math. Phys.}, 399(1):57--83, 2023.

\bibitem{GarbanVargas}
C.~Garban and V.~Vargas.
\newblock Harmonic analysis of {G}aussian multiplicative chaos on the circle.
\newblock Preprint, available at arXiv:2311.04027.

\bibitem{Grafakos08}
L.~Grafakos.
\newblock {\em Classical {F}ourier analysis}, volume 249 of {\em Graduate Texts in Mathematics}.
\newblock Springer, New York, second edition, 2008.

\bibitem{Heurteaux2016}
Y.~Heurteaux.
\newblock An introduction to {M}andelbrot cascades.
\newblock In {\em New trends in applied harmonic analysis}, Appl. Numer. Harmon. Anal., pages 67--105. Birkh\"{a}user/Springer, Cham, 2016.

\bibitem{HK1997}
B.~R. Hunt and V.~Y. Kaloshin.
\newblock How projections affect the dimension spectrum of fractal measures.
\newblock {\em Nonlinearity}, 10(5):1031--1046, 1997.

\bibitem{Jaffuel2012}
B.~Jaffuel.
\newblock The critical barrier for the survival of branching random walk with absorption.
\newblock {\em Ann. Inst. Henri Poincar\'e{} Probab. Stat.}, 48(4):989--1009, 2012.

\bibitem{Kahane1993}
J.-P. Kahane.
\newblock Fractals and random measures.
\newblock {\em Bull. Sci. Math.}, 117(1):153--159, 1993.

\bibitem{KahanePeyriere1976}
J.-P. Kahane and J.~Peyri\`ere.
\newblock Sur certaines martingales de {B}enoit {M}andelbrot.
\newblock {\em Advances in Math.}, 22(2):131--145, 1976.

\bibitem{Lyons1995}
R.~Lyons.
\newblock Seventy years of {R}ajchman measures.
\newblock In {\em Proceedings of the {C}onference in {H}onor of {J}ean-{P}ierre {K}ahane ({O}rsay, 1993)}, pages 363--377, 1995.

\bibitem{Madaule2017}
T.~Madaule.
\newblock Convergence in law for the branching random walk seen from its tip.
\newblock {\em J. Theoret. Probab.}, 30(1):27--63, 2017.

\bibitem{Mandelbrot1974}
B.~Mandelbrot.
\newblock Intermittent turbulence in self similar cascades: divergence of high moments and dimension of carrier.
\newblock {\em J. Fluid Mech.}, 62:331--333, 1974.

\bibitem{Mandelbrot1976}
B.~Mandelbrot.
\newblock Intermittent turbulence and fractal dimension: kurtosis and the spectral exponent {$5/3+B$}.
\newblock In {\em Turbulence and {N}avier-{S}tokes equations ({P}roc. {C}onf., {U}niv. {P}aris-{S}ud, {O}rsay, 1975)}, volume Vol. 565 of {\em Lecture Notes in Math.}, pages 121--145. Springer, Berlin-New York, 1976.

\bibitem{Mattila2015}
P.~Mattila.
\newblock {\em Fourier analysis and {H}ausdorff dimension}, volume 150 of {\em Cambridge Studies in Advanced Mathematics}.
\newblock Cambridge University Press, Cambridge, 2015.

\bibitem{Molchan1996}
G.~M. Molchan.
\newblock Scaling exponents and multifractal dimensions for independent random cascades.
\newblock {\em Comm. Math. Phys.}, 179(3):681--702, 1996.

\bibitem{Ryou2024}
D.~Ryou.
\newblock Near-optimal restriction estimates for {C}antor sets on the parabola.
\newblock {\em Int. Math. Res. Not. IMRN}, (6):5050--5099, 2024.

\bibitem{Sahlsten}
T.~Sahlsten.
\newblock Fourier transforms and iterated function systems.
\newblock Preprint, available at arXiv:2311.00585.

\bibitem{ShmerkinSuomala2018}
P.~Shmerkin and V.~Suomala.
\newblock Spatially independent martingales, intersections, and applications.
\newblock {\em Mem. Amer. Math. Soc.}, 251(1195):v+102, 2018.

\bibitem{Vershynin2018}
R.~Vershynin.
\newblock {\em High-dimensional probability}, volume~47 of {\em Cambridge Series in Statistical and Probabilistic Mathematics}.
\newblock Cambridge University Press, Cambridge, 2018.
\newblock An introduction with applications in data science, With a foreword by Sara van de Geer.

\end{thebibliography}
\end{document}